\documentclass[
final
]{dmtcs-episciences}

\usepackage[utf8]{inputenc}
\usepackage{subfigure}

\usepackage[round]{natbib}

\usepackage{amsmath}
\usepackage{amsfonts}
\usepackage{amssymb}
\usepackage{amsthm}
\usepackage{breqn}
\usepackage{setspace}
\usepackage{bbm}
\usepackage{tikz}
\usepackage{mathtools} 
\usetikzlibrary{patterns,arrows,decorations.pathreplacing}
\usepackage{comment}
\usepackage{hyperref}
\usepackage{comment}

\newcommand{\A}{\mathcal{A}}

\newcommand{\M}{\mathcal{M}}
\newcommand{\N}{\mathcal{N}}

\newcommand{\h}{\mathcal{H}}

\newcommand{\s}{\mathcal{S}}
\newcommand{\abs}[1]{\left\lvert{#1}\right\rvert}
\newcommand{\floor}[1]{\left\lfloor{#1}\right\rfloor}

\DeclareMathOperator{\ex}{ex}

\newtheorem{lemma}{Lemma}
\newtheorem{theorem}{Theorem} 
\newtheorem{corollary}{Corollary} 
\newtheorem{definition}{Definition}
\newtheorem{claim}{Claim}
\newtheorem{conjecture}{Conjecture}

\newtheorem{remark}{Remark}
\newtheorem{proposition}{Proposition}

\author{Ervin Gy\H{o}ri\affiliationmark{1}\thanks{The research was supported by the National Research, Development and Innovation Office grants  K116769, K117879 and K126853.}
  \and Nika Salia\affiliationmark{1,2}\thanks{The research was supported by the National Research, Development and Innovation Office grants  K116769, K117879 and K126853 as well as the Shota Rustaveli National Science Foundation of Georgia SRNSFG, grant number DI-18-118.}
  \and Casey Tompkins\affiliationmark{1,3}\thanks{The research was supported by the National Research, Development and Innovation Office grants  K116769, K117879 and K126853.}
  \and Oscar Zamora\affiliationmark{2,4}}
  
  \title[Formatting an article for DMTCS]{How to format an article
  for DMTCS\\
  with the journal's own \LaTeX-style}
  
\title{The maximum number of $P_\ell$ copies in $P_k$-free graphs}
\affiliation{
 Alfr\'ed R\'enyi Institute of Mathematics, Hungarian Academy of Sciences.\\
 Central European University, Budapest.\\
  Karlsruhe Institute of Technology, Karlsruhe, Germany.\\
  Universidad de Costa Rica, San Jos\'e. }
\received{2018-11-6}

\revised{2019-6-4}

\accepted{2019-6-14}
\begin{document}
\publicationdetails{21}{2019}{1}{14}{4958}
\maketitle
\begin{abstract}
Generalizing Tur\'an's classical extremal problem, Alon and Shikhelman investigated the problem of maximizing the number of copies of $T$ in an $H$-free graph, for a pair of graphs $T$ and $H$.  Whereas Alon and Shikhelman were primarily interested in determining the order of magnitude for some classes of graphs $H$, we focus on the case when $T$ and $H$ are paths, where we find asymptotic and exact results in some cases. We also consider other structures like stars and the set of cycles of length at least $k$, where we  derive asymptotically sharp estimates.  Our results generalize well-known extremal theorems of Erd\H{o}s and Gallai.
\end{abstract}

\section{Introduction}
For a graph $G$, we let $e(G)$ denote the number of edges in $G$, and for a given graph $H$, we let $\N(H,G)$ denote the number of (not necessarily induced) copies of $H$ in $G$. If there is no copy of $H$ in $G$, we say that $G$ is $H$-free.  We denote the path with $k$ edges by $P_k$ and the cycle with $k$ edges by $C_k$.  By $C_{\ge k}$ we mean the set of all cycles of length at least $k$. By $S_k$ we denote the star on $k+1$ vertices.   Given a graph $G$ containing a vertex $v$, we denote the neighborhood of $v$ by $N(v)$.  The independence number, minimum degree and number of vertices in $G$ are denoted by $\alpha(G)$, $\delta(G)$ and $v(G)$, respectively. The vertex and edge sets of $G$ are denoted by $V(G)$ and $E(G)$, respectively.  Finally, given a set $S \subseteq V(G)$, we denote by $G[S]$ the induced subgraph of $G$ with vertex set $S$.

Following the notation of \cite{alon2016many}, we let $\ex(n,T,H)$ be the maximum number of (noninduced) copies of $T$ in an $H$-free graph on $n$ vertices.  Observe that we have $\ex(n,P_1,H) = \ex(n,H)$, the classical extremal number. If a set of graphs $\h$ is forbidden, then we define $\ex(n,\h)$ (and similarly $\ex(n,T,\h)$), in the obvious way.

We begin by recalling the famous theorem of Erd\H{o}s and Gallai on $P_k$-free graphs as well as some recent generalizations due to Luo, where the number of cliques is considered.

\begin{theorem}[\cite{erdHos1959maximal}]
\label{eg}
For all $n \ge k$,
\begin{displaymath}
\ex(n,P_k) \le \frac{(k-1)n}{2},
\end{displaymath}
and equality holds if and only if $k$ divides $n$ and $G$ is the disjoint union of cliques of size $k$.
\end{theorem}
In their paper, Erd\H{o}s and Gallai deduced Theorem \ref{eg} as a corollary of the following result about graphs with no long cycles.  
\begin{theorem}[\cite{erdHos1959maximal}]
For all $n \ge k$,
\begin{displaymath}
\ex(n,C_{\ge k}) \le \frac{(k-1)(n-1)}{2},
\end{displaymath}
and equality holds if and only if $k-2$ divides $n-1$ and $G$ is a connected graph such that every block of $G$ is a clique of size $k-1$.
\end{theorem}

As the extremal examples for Theorem \ref{eg} are disconnected, it is natural to consider a version of the problem where the base graph is assumed to be connected. \cite{kopylov1977maximal} settled this problem, and later 
\cite{balister2008connected} classified the extremal cases. 

\begin{definition}
We denote by $G_{n,k,a}$ the graph whose vertex set is partitioned into 3 classes, $A, B$ and $C$ with $\abs{A}=a$, $\abs{B} = n-k+a$, $\abs{C} =k-2a$ such that $A \cup C$ induces a clique, $B$ is an independent set and all possible edges are taken between vertices of $A$ and $B$. (See Figure \ref{connegfig}.)
\end{definition}

 Throughout this paper we let $t = \floor{\frac{k-1}{2}}$. In $G_{n,k,t}$, the class $C$ has one vertex when $k$ is odd or two vertices when $k$ is even.  By grouping $B$ and $C$ together, we have that $G_{n,k,t}$  is obtained from a complete bipartite graph $K_{t,n-t}$ by adding all edges in the color class of size $t$, and in the case that $k$ is even, adding one additional edge inside the color class of size $n-t$.

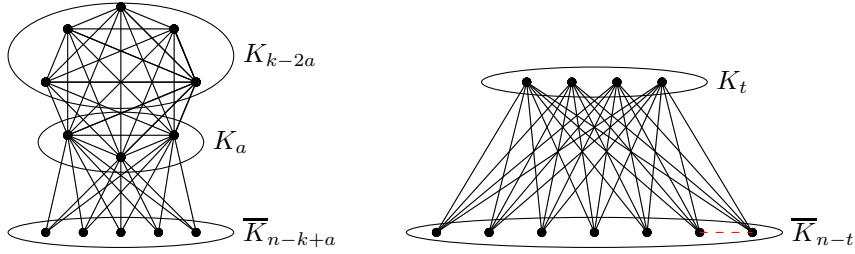
\begin{figure}
\centering
\label{connegfig}

\begin{tikzpicture}

\foreach \xi in {0,45,...,360}{
\pgfmathsetmacro{\z}{\xi+45}
\foreach \yi in {\xi,\z,...,360}{

\filldraw  (xyz polar cs:angle=\xi,radius=1) circle (1.5pt) -- (xyz polar cs:angle=\yi,radius=1) circle (1.5pt);

}
}

\foreach \x in{-1,-0.5,...,1}{
\foreach \y in {225,270,...,315}{

\filldraw (xyz polar cs:angle=\y,radius=1) -- (\x,-2) circle (1.5pt);

}

}

\draw  (1.5,-2) arc (0:360:1.5cm and .2cm) node[align=center,right]{$\overline{K}_{n-k+a}$};

   \draw (1.1,-.8) arc (0:360: 1.1cm and .4cm) node[align=center,right]{$K_{a}$};

\draw (1.5,.35) arc(0:360: 1.5cm and .7cm) node[align=center,right]{$K_{k-2a}$};

\end{tikzpicture} \qquad
\begin{tikzpicture}

\foreach \x in {-2.1,-1.4,...,2.8}{
\foreach \y in {-.9,-.3,...,.9}{

\filldraw (\y,0) circle (1.5pt) -- (\x,-2) circle (1.5pt);

}
}

\draw[red, dashed] (1.4,-2) -- (2.1,-2);

\draw  (2.5,-2) arc (0:360:2.5cm and .2cm) node[align=center,right]{$\overline{K}_{n-t}$};

\draw (1.5,0) arc(0:360: 1.5cm and .2cm) node[align=center,right]{$K_{t}$};

\end{tikzpicture}












\caption{The graph $G_{n,k,a}$ is pictured on the left, and the special case of $G_{n,k,t}$ is pictured on the right.  The dashed edge appears only when $k$ is even.}
\end{figure}


\begin{theorem}[\cite{kopylov1977maximal}, \cite{balister2008connected}]
\label{connEG}
Let $G$ be a connected $n$-vertex $P_k$-free graph, with $n\geq k$, then 
\begin{displaymath}
e(G) \le \max(e(G_{n,k,t}),e(G_{n,k,1})).
\end{displaymath}
Moreover, the extremal graph is either $G_{n,k,t}$ or $G_{n,k,1}$.
\end{theorem}
Note that if $n \geq 5k/4$, this maximum is achieved by $G_{n,k,t}$.  Also observe that 
\begin{displaymath}
e(G_{n,k,t}) = t(n-t) + \binom{t}{2}+ \eta_k,
\end{displaymath}
where $\eta_k$ is $1$, if $k$ is even, and $0$ otherwise. Thus, Theorems \ref{eg} and \ref{connEG} yield the same bound asymptotically as $n$ tends to infinity.

The following theorem was deduced by \cite{luo2017maximum} as a corollary of her main result but also follows from Theorem \ref{eg} using a simple induction argument. We present this proof here.
\begin{theorem}[\cite{luo2017maximum}]
\label{luo}
\begin{displaymath}
\ex(n,K_r,P_k) \le \frac{n}{k} \binom{k}{r}.
\end{displaymath}
\end{theorem}
\begin{proof}
We use induction on $r$, and the base is Theorem \ref{eg}. Let $G$ be an $n$-vertex graph containing no $P_k$. We have
\begin{displaymath}
\N(K_r,G) = \frac{1}{r}\sum_{v\in V(G)}\N(K_{r-1},G[N(v)]) \le \frac{1}{r}\sum_{v\in V(G)}\frac{v(G[N(v)])}{k-1} \binom{k-1}{r-1} = \frac{1}{k(k-1)}\binom{k}{r}2e(G),
\end{displaymath}
since $G[N(v)]$ contains no $P_{k-1}$. By Theorem \ref{eg}, we have $e(G) \le \frac{(k-1)n}{2}$, and the result follows.
\end{proof}
For our results we will need only that $\ex(n,K_r,P_k) \le c_{k,r} n$ for some constant $c_{k,r}$ depending only on $k$ and $r$.

If we impose the additional condition that the graph is connected, then the situation is more complicated.  Luo proved the following sharp bounds.
\begin{theorem}[\cite{luo2017maximum}]
Let $n > k \ge 3$ and $G$ be a connected $n$-vertex graph with no path of length $k$, then
\begin{displaymath}
\N(K_r,G) \le \max\left(\N(K_r,G_{n,k,t}),\N(K_r,G_{n,k,1})\right).
\end{displaymath}
\end{theorem}

\begin{theorem}[\cite{luo2017maximum}]
Let $n \ge k \ge 4$ and $G$ be a $n$-vertex graph with no cycle of length $k$ or greater, then
\begin{displaymath}
\N(K_r,G) \le \frac{n-1}{k-2}\binom{k-1}{r}.
\end{displaymath}
\end{theorem}
Some recent generalizations of the Erd\H{o}s--Gallai theorem and Luo's results can be found in \cite{newLuo}. In the present paper we focus on results where paths or all sufficiently long cycles are forbidden.  The general problem of enumerating cycles of a fixed length when a fixed cycle is forbidden has also been considered recently (see \cite{cycle1} and \cite{cycle2} which generalize earlier results for special cases, e.g., \cite{c3c5},  \cite{c3ckplus1}, \cite{alon2016many}). 


\cite{alon2016many} considered the problem of maximizing the number of copies of a tree $T$ in a graph which is $H$-free, for another tree $H$. Given two trees $T$ and $H$, they introduced an integer parameter $m(T,H)$ and proved that $\ex(n,T,H) = \Theta(n^{m(T,H)})$, thereby determining the correct order of magnitude for all pairs of trees.  A recent result  due to \cite{newtree} extends the above result of Alon and Shikhelman to the case when only $H$ is a tree and $T$ is arbitrary.  It is shown that, nonetheless, the order of magnitude of $\ex(n,T,H)$ is a positive integer power of $n$.

In the present paper, we are interested in the case where the forbidden tree is a path, and we find correct asymptotics and sometimes the exact bound for the maximum number of copies of a smaller path (as well as for several other types of graphs). We also obtain asymptotic results for the problem of maximizing copies of $T$ in a graph with no cycles of length at least $k$, in the case when $T$ is a path. 

The paper is organized as follows:  In Section \ref{asy}, we determine asymptotically the maximum number of paths and cycles in a $P_k$-free graph. For the case when $k$ is even we provide a simple proof using a result of Nikiforov on the spectral radius of $P_k$ free graphs.  Then, we give more precise estimates which are also sharp in case when $k$ is odd through double-counting arguments.  In Section~\ref{treefree}, we determine the order of magnitude of $\ex(n,H,T)$ when $T$ is a tree for the class of graphs $H$ which satisfy the condition that $v(H) - \alpha(H) \le \floor{\frac{k-1}{2}}$.  In Section \ref{exact}, we determine $\ex(n,H,P_k)$ exactly for several graphs $H$ including 4-cycles, stars and short paths. In Section~\ref{weird case}, we consider the problem of enumerating copies of $P_{k-1}$ in a $P_k$-free graph. We determine the asymptotic result for copies of $P_5$ in a $P_6$-free graph and pose a general conjecture.

\section{Asymptotic Results}
 We write $f(n,k) {\sim} g(n,k)$ when $\displaystyle
 \lim_{k \to \infty} \left(\lim_{n \to \infty} \frac{f(n,k)}{g(n,k)}\right) = 1$. 
We estimate on the number of copies of paths and cycles in a $P_k$-free graph. For a fixed $\ell \in \mathbb{N}$, we prove the following asymptotic results:
\label{asy}
\begin{theorem}
\label{evenP}
\begin{displaymath}
\ex(n,P_{2\ell},P_k) \sim \frac{k^\ell n^{\ell+1}}{2^{\ell+1}}.
\end{displaymath} 
\end{theorem}

\begin{theorem}
\label{oddP}
\begin{displaymath}
\ex(n,P_{2\ell+1},P_k) \sim \frac{(\ell+2)k^{\ell+1} n^{\ell+1}}{2^{\ell+2}}.
\end{displaymath} 
\end{theorem}

\begin{theorem}
\label{evenC}
\begin{displaymath}
\ex(n,C_{2\ell},P_k) \sim  \frac{k^\ell n^\ell}{\ell 2^{\ell+1}}.
\end{displaymath}
\end{theorem}
\begin{theorem}
\label{oddC}
\begin{displaymath}
\ex(n,C_{2\ell+1},P_k) \sim \frac{k^{\ell+1} n^\ell}{2^{\ell+2}}. 
\end{displaymath}
\end{theorem}

The construction showing the lower bounds for Theorems \ref{evenP}  through \ref{oddC} is the same as the extremal construction for the connected version of the Erd\H{o}s--Gallai theorem, Theorem \ref{connEG}.  Because we are interested in asymptotics, we will omit one edge from this construction which only occurs when $k$ is even.  Our $n$-vertex graph $G$ is defined by taking a clique on a set $S$ of $\floor{\frac{k-1}{2}}$ vertices and connecting every vertex in $S$ to every vertex of an independent set $U$, defined on $n-\floor{\frac{k-1}{2}}$ vertices.  It is easy to see that this graph is $P_k$-free.  In enumerating the copies of $P_{2\ell}$, the only paths which contribute asymptotically alternate between $S$ and $U$, starting and ending with $U$ (the factor of 2 comes from counting the path in both directions).   

When enumerating the copies of $P_{2\ell+1}$,  we have two kinds of paths which contribute asymptotically: those that start and end in $U$, using an edge in $S$ at some step, and those that start in $U$ and end in $S$, never using an edge contained in $S$.  For the first type, we condition on which step in the path we use the edge in $S$ ($\ell$ possibilities).  Each such path gets counted twice, hence we divide by two.  For the second type, each path is counted once and so we do not have to divide by 2.    

We begin by showing how Theorem~\ref{evenP} can be derived from a result about the spectral radius of $P_k$-free graphs due to \cite{nikiforov2010spectral}. 
Recall that the spectral radius of a graph $G$ is the maximum of the eigenvalues of the adjacency matrix of $G$. 
He determined, for sufficiently large $n$, the maximal spectral radius of a $P_k$-free graph on $n$ vertices. We are interested in asymptotics so we will make use of the following corollary which follows directly from the results in \cite{nikiforov2010spectral}.

\begin{corollary}[\cite{nikiforov2010spectral}] \label{Nikcor}
If $n$ is sufficiently large and $G$ is a $P_k$-free graph, then the spectral radius of $G$ is at most $\sqrt{\lfloor (k+1)/2 \rfloor n}$.
\end{corollary}  
	
\begin{proof}[of Theorem \ref{evenP} using spectral bounds]
Let $G$ be a $P_k$-free graph on $n$ vertices (for $n$ large enough to satisfy Corollary \ref{Nikcor}).  Let $A$ be the adjacency matrix of $G$, then we have
\begin{displaymath}
2\cdot \frac{\N(P_{2\ell},G)}{n} \leq  \frac{\#2\ell\mbox{-walks in 
}G}{n} = \frac{\mathbf{1}^t A^{2\ell} \mathbf{1}}{\mathbf{1}^t \mathbf{1}} \leq \left(\sqrt{\lfloor (k+1)/2 \rfloor n}\right)^{2\ell} = (\lfloor (k+1)/2 \rfloor n)^\ell.
\end{displaymath}
Where $\mathbf{1}$ is the all $1$'s vector, and the second inequality comes from the fact that the spectral radius of a Hermitian matrix $M$ is the supremum of the quotient $\frac{x^*Mx}{x^*x}$, where $x$ ranges over $\mathbb{C}^n \backslash \{0\}$. 
Therefore, for every $k\in \mathbb{N}$ and $n$ sufficiently large we have  $\ex(n,P_{2\ell},k) \leq n^{\ell+1}\lfloor{(k+1)/2} \rfloor^\ell/2$. 
\end{proof}

Unfortunately, it does not seem like this approach can be used to prove Theorem \ref{oddP} as the bound it would yield is off by a factor of order $\sqrt{n}$.

We will now prove the upper bounds from which Theorems \ref{evenP} and \ref{oddP} are immediate consequences. We note that the upper bound we obtain for the $P_{2\ell}$-case is sharper than the proof using the spectral radius yields.

\begin{proposition}
Let $\ell,k$ be positive integers with $2\ell < k$, then
\begin{displaymath}
\ex(n,P_{2\ell},P_k) \leq \frac{k^\ell n^{\ell+1}}{2^{\ell+1}} + O(n^\ell).
\end{displaymath}

\end{proposition}

\begin{proposition}
Let $\ell,k$ be positive integers with $2\ell+1 < k$, then 
\begin{displaymath}
\ex(n,P_{2\ell+1},P_k) \leq \frac{(\ell+2)k^{\ell+1} n^{\ell+1}}{2^{\ell+2}} + O(n^\ell).
\end{displaymath}

\end{proposition}

The proofs of the propositions above will use a double-counting argument involving structures defined using matchings.  We will begin by estimating the maximum number of certain kinds of matchings occurring in a $P_k$-free graph.

Let us define $M_1^{\ell}$, $M_2^{\ell}$ and $M_3^{\ell}$ to be the following graphs: $M_1^{\ell}$ is an $(\ell-1)$-matching together with a disjoint triangle, $M_2^{\ell}$ is an $(\ell-1)$-matching together with a disjoint $K_4$ and  $M_3^{\ell}$ is an $(\ell-2)$-matching with two independent triangles, disjoint from the matching (see Figure \ref{negmatch}).

\begin{figure}[b]
\centering
\label{negmatch}
\begin{tikzpicture}
\filldraw (0,0) circle (2pt) -- (0,1) circle (2pt);
\filldraw (.5,0) circle (2pt) -- (.5,1) circle (2pt);

\filldraw (1.5,0) circle (2pt) -- (1.5,1) circle (2pt);
\filldraw (2,0) circle (2pt) -- (2,1) circle (2pt);
\filldraw (2,0) circle (2pt) -- (2.5,.5) circle (2pt)  -- (2,1) circle (2pt);

\draw (1.25,1.2) node[align=center, above]  {$M_1^{\ell}$};

\draw (1,0) node[align=center]{$\cdots$};
\draw [decorate,decoration={brace,amplitude=10pt,mirror},yshift=-4pt](0,0) -- (1.5,0) node[black,midway,yshift=-0.6cm] {\footnotesize $\ell-1$};


\filldraw (4,0) circle (2pt) -- (4,1) circle (2pt);
\filldraw (4.5,0) circle (2pt) -- (4.5,1) circle (2pt);

\filldraw (5.5,0) circle (2pt) -- (5.5,1) circle (2pt);
\filldraw (6,0) circle (2pt) -- (6,1) circle (2pt);
\filldraw (6.5,0) circle (2pt) -- (6.5,1) circle (2pt);

\draw (6,0)  -- (6.5,0)  -- (6,1) ;
\draw (6,0)  -- (6.5,1)  -- (6,1) ;

\draw (5,1.2) node[align=center, above]  {$M_2^{\ell}$};

\draw (5,0) node[align=center]{$\cdots$};
\draw [decorate,decoration={brace,amplitude=10pt,mirror},yshift=-4pt](4,0) -- (5.5,0) node[black,midway,yshift=-0.6cm] {\footnotesize $\ell-1$};

\filldraw (8,0) circle (2pt) -- (8,1) circle (2pt);
\filldraw (8.5,0) circle (2pt) -- (8.5,1) circle (2pt);

\filldraw (9.5,0) circle (2pt) -- (9.5,1) circle (2pt);
\filldraw (10,0) circle (2pt) -- (10,1) circle (2pt);
\filldraw (10,0) circle (2pt) -- (10.5,0) circle (2pt) -- (10,1) circle (2pt);
\filldraw (11,0) circle (2pt) -- (10.5,1) circle (2pt) -- (11,1) circle (2pt);
\filldraw (11,0) circle (2pt) -- (11,1) circle (2pt);
\draw (9.5,1.2) node[align=center, above]  {$M_3^{\ell}$};

\draw (9,0) node[align=center]{$\cdots$};
\draw [decorate,decoration={brace,amplitude=10pt,mirror},yshift=-4pt](8,0) -- (9.5,0) node[black,midway,yshift=-0.6cm] {\footnotesize $\ell-2$};

\end{tikzpicture}
\caption{Matching structures with negligible contribution.}
\end{figure}
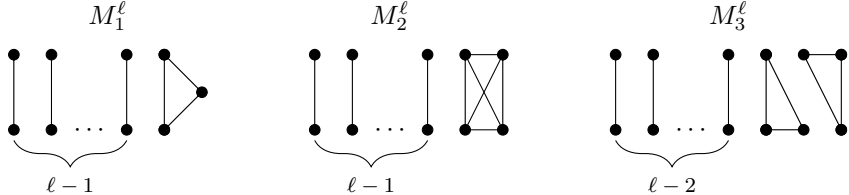

\begin{lemma}
\label{matching}
The number of copies of $M_1^{\ell}$, $M^{\ell}_2$ and $M^{\ell}_3$  in an $n$-vertex $P_k$-free graph is $O(n^{\ell})$.  
\end{lemma}
\begin{proof}
Let $G$ be a $P_k$-free graph on $n$ vertices.  By Theorem \ref{luo}, the number of triangles in $G$ is at most $O(k^2n)$. By Theorem \ref{eg} the total number of edges in $G$ is at most $(k-1)n/2$.  It follows that the number of copies of $M_1^\ell$ is bounded from above by 
\begin{displaymath}
\binom{\frac{kn}{2}}{\ell-1} k^2n = O(k^{\ell+1}n^\ell). 
\end{displaymath}
The proofs of the bound for $M^{\ell}_2$ and $M^{\ell}_3$ are similar.
\end{proof}

\begin{proof}[of Theorem \ref{evenP}]
Let $G$ be a $P_k$-free graph on $n$ vertices.  We will consider structures consisting of a matching of $\ell$ edges and a vertex not contained in these edges.  Namely, a \emph{matching structure} is an $(\ell+1)$-tuple $(e_1,e_2,\dots,e_\ell,v)$ where $\{e_1,e_2,\dots,e_\ell\}$ is a matching in $G$ and $v \in V(G) \setminus \cup_{i=1}^\ell e_i$.  We say that a path $P_{2\ell}$ \emph{aligns} with a matching structure $(e_1,e_2,\dots,e_\ell,v)$ if its edges are (consecutively) $e_1, f_1, e_2, f_2, \dots,  e_{\ell}, f_{\ell}$ where $v \in f_\ell$. We say that the matching structure \emph{spans} the set of vertices $\cup_{i=1}^\ell e_i \cup \{v\}$.

Let $\A \coloneqq \{S\subseteq V: \abs{S} = 2\ell +1, M_1^{\ell}\subseteq G[S]\}$. By  Lemma \ref{matching}, we have $\abs{\A} = O(n^{\ell})$. Let $\M$ be the set of all the matching structures which span a set of vertices not contained in $\A$.

\begin{claim}
\label{firstclaim}
At most one $P_{2\ell}$ aligns with each matching structure in $\mathcal{M}$.
\end{claim}

\begin{proof}
Let $(e_1,e_2,\dots,e_\ell,v)$ be a matching structure in $\mathcal{M}$ and fix a $P_{2\ell}$ which aligns with it, say \\ $a_1,b_1,a_2,b_2\dots,a_{\ell},b_\ell,a_{\ell+1}$, where $e_i = \{a_i,b_i\}$ and $v=a_{\ell+1}$.  There is no edge from $a_i$ to $a_{i+1}$, otherwise $e_1,e_2,\dots,e_{i-1}, \{b_{i+1},a_{i+2}\},\{b_{i+2},a_{i+3}\},\dots,\{b_{\ell},a_{\ell+1}\}$ together with the triangle $\{a_i,b_i,a_{i+1}\}$ forms an $M_1^{\ell}$.  Since there is a unique $P_{2\ell}$ spanning the matching structure and not containing an edge $\{a_i,a_{i+1}\}$, the claim is proved. (See Figure \ref{bob}.)
\end{proof}

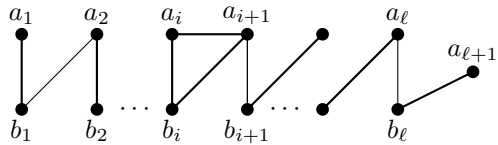
\begin{figure}[b]
\centering
\begin{tikzpicture}
\filldraw[thick] (0,0) node[align=center, below]{$b_1$} circle (2pt) -- (0,1) node[align=center, above]{$a_1$} circle (2pt);
\draw (0,0) -- (1,1);
\filldraw[thick] (1,0) node[align=center, below]{$b_2$} circle (2pt) -- (1,1) node[align=center, above]{$a_2$}  circle (2pt);
\draw (1.5,0) node[align=center]{$\cdots$};

\filldraw[thick] (2,0) node[align=center, below]{$b_i$} circle (2pt) -- (2,1) node[align=center, above]{$a_i$} circle (2pt);
\filldraw[thick] (2,0) -- (3,1) node[align=center, above]{$a_{i+1}$} circle (2pt) -- (2,1) ;
\filldraw (3,0)   -- (3,1) ;
\filldraw[thick] (3,0) node[align=center, below]{$b_{i+1}$} circle (2pt) -- (4,1) circle (2pt);

\draw (3.5,0) node[align=center]{$\cdots$};

\draw (5,0) -- (5,1);
\filldraw[thick] (5,1) node[align=center, above]{$a_{\ell}$} circle (2pt) -- (4,0) circle (2pt);
\filldraw[thick] (5,0)  node[align=center, below]{$b_\ell$} circle (2pt) -- (6,.5) circle (2pt) node[align=center, above]{$a_{\ell+1}$} ;

\end{tikzpicture}
\caption{Matching structure from the proof of Claim \ref{firstclaim}.}
\label{bob}
\end{figure}


Next, we observe that for every $P_{2\ell}$, there are precisely two matching structures for which that $P_{2\ell}$ is aligned.  Indeed, let the vertices of the $P_{2\ell}$ be traversed in the order $v_1,v_2,\dots,v_{2\ell+1}$, then the two matching structures with which the $P_{2\ell}$ aligns are 
\begin{displaymath}
(\{v_1,v_2\},\{v_3,v_4\},\dots,\{v_{2\ell-1},v_{2\ell}\},v_{2\ell+1}) \mbox{ and }
(\{v_{2\ell+1},v_{2\ell}\},\{v_{2\ell-1},v_{2\ell-2}\},\dots,\{v_3,v_2\},v_1).
\end{displaymath} 
It follows that the if we define $M \coloneqq \abs{\mathcal{M}}$, then the number of copies of $P_{2\ell}$ is bounded from above by $M/2 + O(n^\ell)$.  

By Theorem \ref{eg}, the number of edges in $G$ is at most $(k-1)n/2$.  A matching structure is formed by choosing $\ell$ edges in order followed by an additional vertex.  Thus, we have the following upper bound on the number of matching structures in $\mathcal{M}$:
\begin{displaymath}
M \le \binom{\frac{nk}{2}}{\ell}\ell!n.
\end{displaymath}
Dividing by $2$ yields the required bound on the number of copies of $P_{2\ell}$.
\end{proof}

\begin{proof}[of Theorem \ref{oddP}]
We will now define matching structures in a slightly different way.  A \emph{matching structure} is an $(\ell+1)$-tuple $(e_1,e_2,\dots,e_{\ell+1})$, where $\{e_1,e_2,\dots,e_{\ell+1}\}$ is a matching in $G$.  A path $P_{2\ell+1}$ \emph{aligns} with a matching structure $(e_1,e_2,\dots,e_{\ell+1})$ if its edges are $e_1, f_1, e_2, f_2, \dots, e_\ell,  f_\ell, e_{\ell+1}$, consecutively.

Let $\mathcal{B} \coloneqq \{S\subseteq V: \abs{S} = 2\ell +2, M_2\subseteq G[S]\}$ and $\mathcal{C} \coloneqq \{S\subseteq V: \abs{S} = 2\ell +2, M_3\subseteq G[S]\}$. By Lemma \ref{matching}, we have $ \abs{\mathcal{B}} = O(n^{\ell})$ and  $\abs{\mathcal{C}} = O(n^{\ell})$. Let $\mathcal{M}$ be the set of matching structures which do not span a vertex set in $\mathcal{B}$ or $\mathcal{C}$.

\begin{claim}
There are at most $\ell+2$ copies of $P_{2\ell+1}$ which align with each matching structure in $\mathcal{M}$.
\end{claim}
\begin{proof}
Consider a matching structure $(e_1,e_2,\dots,e_{\ell+1}) \in \mathcal{M}$.   We will consider the edges in the matching structure one by one and show that we can label the vertices of each edge $e_j$ with $a_j$ and $b_j$ in such a way that there is no edge between $a_j$ and $a_{j+1}$. Thus, every path which aligns with the matching structure will be a subgraph of the graph pictured (on the top) in Figure \ref{goal}.  Given that the matching structure has this form, we may easily upper bound the number of copies of $P_{2\ell+1}$ which can align with it.  Indeed, if the $P_{2\ell+1}$ starts with the vertex $b_1$, there is at most one such path:  $b_1, a_1, b_2, a_2, \dots, b_{\ell+1}, a_{\ell+1}$.  If it starts with the vertex $a_1$, then for at most one $i$, $1 \le i \le \ell$, the path may use an edge $\{b_i,b_{i+1}\}$; all other choices are forced. Thus, in total there are at most $1 + (\ell+1)= \ell+2$ paths which align with such a matching structure. We now prove that the desired labeling of the edges exists.

We may suppose there is at least one edge from $e_i$ to $e_{i+1}$ for all $i=1,2,\dots,\ell$, otherwise no $P_{2\ell+1}$ aligns with the matching structure. We also know $e_i \cup e_{i+1}$ does not induce a $K_4$, so there is at least one edge missing among these 4 vertices. Now we may define $e_1 = \{a_1,b_1\}$ in such a way that there is at least one edge missing from $a_1$ to $e_2$.  Define $e_2=\{a_2,b_2\}$ such that there is no edge between $a_1$ and $a_2$. In general, suppose we have already labeled the edges $e_1,e_2,\dots,e_j$ in such a way that for $i\in\{1,2,\dots,j-1\}$, $a_i$ is not connected to $a_{i+1}$. We will show that $e_{j+1}$ can be labeled by $a_{j+1}$ and $b_{j+1}$ such that there is no edge between $a_j$ and $a_{j+1}$.

We know there is an edge missing from $e_j$ to $e_{j+1}$. If there is an edge missing between $a_j$ and $e_{j+1}$, then define $e_{j+1}=\{a_{j+1},b_{j+1}\}$ so that there is an edge from $a_j$ to $a_{j+1}$.  Otherwise $\{a_j\}\cup e_{j+1}$ forms a triangle. In this case, there is an edge missing from $b_j$ to $e_{j+1}$; label $e_{j+1}=\{a_{j+1},b_{j+1}\}$ so that $b_j$ is not adjacent to $b_{j+1}$.

Now if we do not have an edge from $a_{j-1}$ to $b_j$, then we switch the labels on $e_j$ and $e_{j+1}$, and we are done. (By switching the labels we mean that the vertex in $e_i$ previously labeled $a_i$ is now labeled $b_i$, and the vertex previously labeled $b_i$ is now labeled $a_i$.)  Thus, assume we have an edge from $a_{j-1}$ to $b_j$.  Then we have no edge from $b_{j-1}$ to $b_j$, for this would yield an $M_3$.  
Next, consider $e_{j-2}$. If there is no edge from $a_{j-2}$ to $b_{j-1}$, then switch the labels on $e_{j-1}$, $e_j$ and $e_{j+1}$, and we are done.  If there is an edge from $a_{j-2}$ to $b_{j-1}$, we proceed similarly with $e_{j-3}$.  Continuing this procedure we will reach an edge $e_r$ such that switching the labels of $e_r, e_{r+1},\dots,e_{j+1}$ yields no edge between $a_i$ and $a_{i+1}$ for any $1 \le i \le j$. (This procedure is illustrated in Figure \ref{goal}.) 
\end{proof}




\begin{figure}[b]
\centering
\begin{tikzpicture}
\filldraw (0,0) circle (2pt) -- (0,1) circle (2pt);
\filldraw (1,0) circle (2pt) -- (1,1) circle (2pt);
\draw  (0,1) -- (1,0) -- (0,0) -- (1,1);

\filldraw (2,0) circle (2pt) -- (2,1) circle (2pt);
\draw  (2,1) -- (1,0) -- (2,0) -- (1,1);
\draw (2.2,0) node[align=center,right]{$\dots$};

\filldraw (3,0) circle (2pt) -- (3,1) circle (2pt);
\filldraw (4,0) circle (2pt) -- (4,1) circle (2pt);
\draw  (3,1) -- (4,0) -- (3,0) -- (4,1);

\draw [decorate,decoration={brace,amplitude=10pt,mirror},yshift=-4pt](0,0) -- (4,0) node[black,midway,yshift=-0.6cm] {\footnotesize $\ell+1$};

\end{tikzpicture}

\begin{tikzpicture}

\draw[red,dashed] (0,1) -- (1,1);
\filldraw[thick] (0,0) node[align=center, below]{$b_1$} circle (2pt) -- (0,1) node[align=center, above]{$a_1$} circle (2pt);

\filldraw[thick] (1,0) node[align=center, below]{$b_2$} circle (2pt) -- (1,1) node[align=center, above]{$a_2$}  circle (2pt);

\draw[thick] (1.5,0) node[align=center]{$\cdots$};

\filldraw[thick] (2,0) node[align=center, below]{$b_{r-1}$} circle (2pt) -- (2,1) node[align=center, above]{$a_{r-1}$} circle (2pt);
\draw [red,dashed] (2,1) -- (3,1) ;
\draw [red,dashed] (5,1) -- (4,1);
\filldraw[thick] (3,1) node[align=center, above]{$a_{r}$} circle (2pt) -- (4,0) circle (2pt);
\filldraw[thick] (3,0) node[align=center, below]{$b_{r}$} circle (2pt) -- (2,1) circle (2pt);
\draw[thick] (3,0) -- (2,0);

\draw (4.5,0) node[align=center]{$\cdots$};
\draw[red,dashed] (5,0)--(6,0);

\filldraw[thick] (4,1) circle (2pt) -- (5,0) circle (2pt) node[align=center, below]{$b_{j}$} ;
\filldraw[thick] (6,1) node[align=center, above]{$a_{j+1}$} circle (2pt) -- (5,1) node[align=center, above]{$a_{j}$} circle (2pt);
\filldraw[thick] (6,0)  node[align=center, below]{$b_{j+1}$} circle (2pt) -- (6,1) circle (2pt)  ;
\draw[thick] (5,1) -- (6,0);
\filldraw[thick] (7,0) circle (2pt) -- (7,1) circle (2pt);
\filldraw[thick] (8,0) circle (2pt)  -- (8,1) circle (2pt);
\draw (7.5,0) node[align=center]{$\cdots$};
\end{tikzpicture}
\caption{The structure of paths aligning with matching structures from $\mathcal{M}$.}
\label{goal}
\end{figure}
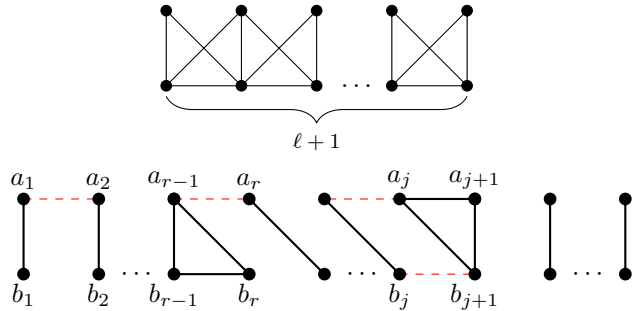


We now complete the proof of Theorem \ref{oddP}. Again we set $M \coloneqq \abs{\mathcal{M}}$.  By Theorem \ref{eg}, there are at most $(k-1)n/2$ total edges in $G$.  Thus,
\begin{displaymath}
\abs{\mathcal{M}}\le \binom{\frac{nk}{2}}{\ell+1}(\ell+1)!.
\end{displaymath}
Since at most $\ell+2$ paths $P_{2\ell+1}$ align with each matching structure from $\mathcal{M}$, and every $P_{2\ell+1}$ aligns with precisely two matching structures.  It follows that the total number of copies of $P_{2\ell+1}$ in $G$ is at most
\begin{displaymath}
\frac{(\ell+2)M}{2} + O(n^\ell) \le \frac{(\ell+2)\binom{\frac{nk}{2}}{\ell+1}(\ell+1)!}{2} + O(n^\ell) = \frac{(\ell+2)k^{\ell+1} n^{\ell+1}}{2^{\ell+2}} + O(n^\ell). \qedhere
\end{displaymath}
\end{proof}


The lower bound for Theorems \ref{evenC}  and \ref{oddC} also comes from $G_{n,k,t}$. Similarly as before, the upper bounds are consequence of the following propositions.

\begin{proposition} \label{p111}
Let $2\ell < k$, then 
\begin{displaymath}
\ex(n,C_{2\ell},P_k) \leq \frac{k^\ell n^{\ell}}{\ell 2^{\ell+1}} + O(n^{\ell-1}).
\end{displaymath}

\end{proposition}

\begin{proposition}
Let $2\ell+1 < k$, then 
\begin{displaymath}
\ex(n,C_{2\ell+1},P_k) \leq \frac{k^{\ell+1} n^{\ell}}{2^{\ell+2}} + O(n^{\ell-1}).
\end{displaymath}
\end{proposition}

It is enough to prove the following claims from which the propositions above follow (a proof of this implication is included after the proof of the claim). Their proofs are similar, so we just give the proof of the first claim.  

\begin{claim}
\label{diffclaim1}
For every $k,\ell \in \mathbb{N}$, there exists $n_0\in\mathbb{N}$ such that if $n \geq n_0$,
\begin{displaymath}
\ex(n+1,C_{2\ell},P_k)-\ex(n,C_{2\ell},P_k) \leq \frac{k^{\ell}n^{\ell-1}}{2^{\ell+1}} + O(n^{\ell-2}).
\end{displaymath}
\end{claim}

\begin{claim}
\label{diffclaim2}
For every $k,\ell \in \mathbb{N}$, there exists $n_0\in\mathbb{N}$ such that if $n \geq n_0$,
\begin{displaymath}
\ex(n+1,C_{2\ell+1},P_k)-\ex(n,C_{2\ell+1},P_k) \leq \frac{\ell k^{\ell+1}n^{\ell-1}}{2^{\ell+2}} + O(n^{\ell-2}).
\end{displaymath}
\end{claim}

\begin{proof}[Proof of Claim \ref{diffclaim1}] 
Let $G$ be a $P_k$-free graph on $n+1$ vertices with maximum number of copies of $C_{2\ell}$.  If $\delta(G) > t$, then by the classical argument of \cite{dirac1952some}, every connected component must have size at most $k$, and then $\N(C_{2\ell},G) \leq k^{2\ell-1}n$. So assume $\delta(G)\leq t$, and let $v$ be a vertex of minimum degree.  Then every $C_{2\ell}$ using $v$ can be divided into two paths: $v$ together with the vertex preceding it and following it in the cycle (forming a $P_2$), and the remaining $2\ell-3$ vertices (forming a $P_{2(\ell-2)}$). Note that every $P_2$ and $P_{2(\ell-2)}$ can be joined in at most two ways to make a $C_{2\ell}$; therefore, the number of copies of $C_{2\ell}$ containing $v$ is at most 
\begin{displaymath}
2{d(v) \choose 2}\ex(n,P_{2(\ell-2)},P_k) \leq 2{t \choose 2}\frac{k^{\ell-2}n^{\ell-1}}{2^{\ell-1}} + O(n^{\ell-2}) \leq \frac{k^{\ell}n^{\ell-1}}{2^{\ell+1}} + O(n^{\ell-2}). \qedhere
\end{displaymath}
\end{proof}

We include a proof that Proposition \ref{p111} follows from Claim \ref{diffclaim1}.  Similar ideas are used throughout the paper.

\begin{proof}[Proof that Claim \ref{diffclaim1} implies Proposition \ref{p111}]
We have 
\begin{align*}
\ex(n,C_{2\ell},P_k) &= \ex(n_0,C_{2\ell+1},P_k)+ \sum_{s=n_0+1}^n\left(\ex(s,C_{2\ell+1},P_k)-\ex(s-1,C_{2\ell+1})\right)\\
&\le \ex(n_0,C_{2\ell+1},P_k)+\frac{k^{\ell}}{2^{\ell+1}} \sum_{s=1}^n \left(s^{\ell-1} + O(n^{\ell-2})\right)\\
&\le \ex(n_0,C_{2\ell+1},P_k) + \frac{k^{\ell}}{2^{\ell+1}} \sum_{s=1}^n \left(\frac{(s+1)^\ell}{\ell}-\frac{s^\ell}{\ell}\right) + O(n^{\ell-1})
\\
&\le \frac{k^\ell n^{\ell}}{\ell 2^{\ell+1}} + O(n^{\ell-1}),
\end{align*}
where in the first inequality we used Claim~\ref{diffclaim1} and pulled the constant out of the sum, and the second inequality follows from $(s+1)^{\ell} = s^{\ell} + \ell s^{\ell-1} + O(s^{\ell-2})$. The final inequality follows from the telescoping sum.  
\end{proof}


\section{The number of copies of $H$ in graphs without a certain tree}
\label{treefree}

Alon and Shikhelman, while considering the case when $H$ is a bipartite graph and $T$ is a tree, mention that $\ex(n,H,T) = O(n^{\alpha(H)})$ is a consequence of a theorem from \cite{alon1981number}. We prove that, in fact, this  holds for general graphs $H$. 

\begin{theorem} \label{general} Let $H$ be any graph and let $T$ be any tree, then $\ex(n,H,T) = O(n^{\alpha(H)}).$
\end{theorem}

\begin{corollary}\label{corcor}
For any graph $H$ such that $v(H) - \alpha(H) \leq \lfloor \frac{k-1}{2} \rfloor$, we have $\ex(n,H,P_k) = \Theta(n^{\alpha(H)})$. 
\end{corollary}
A construction yielding the lower bound in Corollary \ref{corcor} is $G_{n,k,t}$. Indeed, for every subset of size $\alpha(H)$ of the independent set in $G_{n,k,t}$ we can find a copy of $H$ by joining the $t = \floor{\frac{k-1}{2}}$ vertices involved in the clique in $G_{n,k,t}$. 

Theorem \ref{general} follows as a simple consequence of the following lemma which will be proven by induction on $\alpha(H)$.


\begin{lemma}
\label{dif}
For any graph $H$ and any tree $T$, \begin{displaymath}\ex(n+1,H,T) - \ex(n,H,T) = O(n^{\alpha(H)-1}).\end{displaymath} Here, the constant given by the $O$ notation depends only on $H$ and $T$.
\end{lemma}

We start by proving the following well-known fact.

\begin{proposition}
\label{simpleprop}
Let $H$ be a graph, and let $u$ be a vertex of $H$. If $H'$ is the graph obtained by removing $u$ together with its neighborhood, then $\alpha(H') \leq \alpha(H)-1.$ 
\end{proposition}

\begin{proof}
If $X$ is a maximal independent set in $H'$, then since no neighbor of $u$ is in $X$, the set $X\cup\{u\}$ is independent in $H$ and so $\alpha(H') + 1 \leq \alpha(H).$
\end{proof}

We are now ready to prove Lemma \ref{dif}.

\begin{proof}[Proof of Lemma \ref{dif}]
For the base case of the induction, note that if $\alpha(H) = 1$, then $H$ is a clique and it is  easy to see that $\ex(n,K_s,T) = O(n)$ for any $s$ and $T$. (We may, for example, use the simple bound of $\ex(n,T) \le v(T) n$, for any tree $T$, and apply an induction argument similar to the proof of Theorem \ref{luo}.)

 To estimate $\ex(n+1,H,T) - \ex(n,H,T)$, we will start with a graph $G$ on $n+1$ vertices which is $T$-free with maximum number of copies of $H$. We know that $\delta(G)  < v(T)$, otherwise $T \subseteq G$. Let $v$ be a vertex of minimum degree in $G$, and we will count the number of copies of $H$ in $G$ containing $v$ as a vertex.

Let $V(H) = \{u_1,u_2,\dots,u_{v(H)}\}$, and let $H_i$ be the graph obtained by removing $u_i$  together with its neighbors. By Proposition \ref{simpleprop}, we know that $\alpha(H_i) \leq \alpha(H)-1$.  Now for each copy of $H$ using $v$ as a vertex, $v$ must play the role of some $u_i$, and the neighbors of $u_i$ must be embedded in the neighborhood of $v$.  Then the other vertices of $H$, that is the vertices of $H_i$, must be embedded in some way in the remaining vertices of $G$.  We have to choose $d_H(u_i)$ vertices in $N(v)$, so the number of copies of $H$ using $v$ is at most 
\begin{displaymath}
\sum_{i=1}^{v(H)} d(v)^{d_H(u_i)}\N(H_i,G) \leq \sum_{i=1}^{v(H)} v(T)^{d_H(u_i)}\N(H_i,G) = \sum_{i=1}^{v(H)}O_{{H_i}}(n^{\alpha(H_i)}) = O(n^{\alpha(H)-1}).  
\end{displaymath}

Thus, if $G'$ is the graph obtained from $G$ by removing $v$, we have that 
\[\ex(n+1,H,T) = \N(H,G) = \N(H,G') + O(n^{\alpha(H)-1}) \leq \ex(n,H,T) + O(n^{\alpha(H)-1}).\] \qedhere 
\end{proof}


For some particular graphs $H$, by studying more carefully the number of copies of $H$ that use some fixed vertex, we can find a better recursion than the one from Lemma \ref{dif}. In the following section, we improve the recursion for several specific classes of graphs.  For these graphs we will find an integer valued function $f(n)$ which is a lower bound of the extremal number $\ex(n,H,T)$, such that $f(n)$ grows faster than $\ex(n,H,T)$ (when they do not agree). Since both functions are integer valued they must coincide eventually.

\section{Exact Results}
\label{exact}
We now turn our attention to proving some exact results. Recall that we are using the notation $t=\lfloor \frac{k-1}{2} \rfloor$. 
\subsection{Number of copies of $C_4$}
We begin by determining the maximal number of copies of $C_4$ in a $P_k$-free graph.
\begin{theorem}
\label{c4pk}
For every integer $k\geq 5$, there exists $n_1 \in \mathbb{N}$ such that if $n \geq n_1$,

\begin{displaymath}
\ex(n,C_4,P_k) = \N(C_4,G_{n,k,t})  ={n-t \choose 2}{t\choose 2} + 3(n-t){t \choose 3} + 3{t \choose 4} + 2\eta_k{t\choose 2},
\end{displaymath}
where $\eta_k = 1$, if $k$ is even, and 0 otherwise. Moreover, the only extremal graph is $G_{n,k,t}$.

\end{theorem}

To prove Theorem \ref{c4pk}, we will prove the following claim from which the theorem follows simply by induction on $n$.


\begin{claim}
\label{clxyz}
There exists $n_0\in\mathbb{N}$ such that if $n \geq n_0$, $\ex(n+1,C_4,P_k) - \ex(n,C_4,P_k) \leq  {t \choose 2}(n-2)$. Equality can hold for this difference only if the unique extremal graph with $n+1$ vertices is $G_{n+1,k,t}$.
\end{claim}

It is easy to see that $\N(C_4,G_{n+1,k,t}) = \N(C_4,G_{n,k,t}) +  {t \choose 2}(n-2)$.  By Claim \ref{clxyz}, $\ex(n+1,C_4,P_k) \le  \ex(n,C_4,P_k) + {t \choose 2}(n-2)$ with equality only if the unique extremal graph with $n+1$ vertices is $G_{n+1,k,t}$. It follows that $\ex(n+1,C_4,P_k) - \N(C_4,G_{n+1,k,t}) \le \ex(n,C_4,P_k)-\N(C_4,G_{n,k,t})$ and so the sequence $\ex(n,C_4,P_k)-\N(C_4,G_{n,k,t})$ is a non-increasing sequence of non-negative integers that is strictly decreasing after every non-zero term.   Thus, this sequence is eventually the constant 0 sequence, and hence, moreover, $G_{n,k,t}$ is eventually the unique extremal graph.  
 
We now prove Claim \ref{clxyz}.  


\begin{proof}
Let $G$ be a $P_k$-free graph on $n+1$ vertices with the maximum number of copies of $C_4$; that is, $\N(C_4,G) = \ex(n+1,C_4,P_k)$.

If $\delta(G) > t$, then by the classical argument of  \cite{dirac1952some}, every connected component must have size at most $k$, and therefore $N(C_4,G) \leq \frac{n+1}{k}3\binom{k}{4} = \frac{(n+1)(k-1)(k-2)(k-3)}{8}$. Then we can choose $n_0$ so that this number is less than $\N(C_4,G_{n,k,t})$ for $n \geq n_0$. Thus, we can assume $\delta(G) \leq t.$

Let $v$ be a vertex of minimum degree. By removing $v$, we are removing at most ${d(v) \choose 2}(n-2) \leq {t \choose 2}(n-2)$ copies of $C_4$.  Equality can hold only if $d(v) = t$ and the neighbors of $v$ have full degree. It follows that if equality holds, then $G$ contains a complete bipartite graph with color classes of size $t$ and $n+1-t$ respectively such that the size $t$ class is a clique.  If $k$ is odd, we have that $G = G_{n+1,k,t}$. If $k$ is even, since $G$ contains the maximum number of $C_4$'s, it follows that $G$ has an additional edge (it cannot have 2 more for otherwise we would have a $P_k$).  Thus, if $k$ is even we also have $G = G_{n+1,k,t}$.


Therefore, either $G = G_{n+1,k,t}$ or by removing a minimum degree vertex $v$ we obtain a graph $G'$ with $\N(C_4,G') > \N(C_4,G) - {t \choose 2}(n-2) = \ex(n+1, C_4,P_k)-  {t \choose 2}(n-2)$. Since $\ex(n, C_4,P_k)\geq  \N(C_4,G')$, we have that $\ex(n+1, C_4,P_k) - \ex(n, C_4,P_k) < {t \choose 2}(n-2).$
\end{proof}

The same argument proves the following. 

\begin{theorem}
For every positive integer $k\geq 5$, there exists $n_1 \in \mathbb{N}$ such that if $n \geq n_1$
\begin{displaymath}
\ex(n,C_4,C_{\geq k}) = \N(C_4,G_{n,k,t})  = {n-t \choose 2}{t\choose 2} + 3(n-t){t \choose 3} + 3{t \choose 4} + 2\eta_k{t\choose 2}, 
\end{displaymath}
where $\eta_k = 1$, if $k$ is even, and 0 otherwise. Moreover, the only extremal graph is $G_{n,k,t}$.
\end{theorem}

\subsection{Number of copies of $S_r$}

We will prove the following theorem about the number of copies of $P_2$.  However, it will follow as a consequence of a more general result about stars.
\begin{theorem} 
For every positive integer $k\geq 3$, there exists $n_1 \in \mathbb{N}$ such that if $n \geq n_1$
\begin{displaymath}\ex(n,P_2,P_k) = \N(P_2,G_{n,k,t}) = t{n - 1\choose 2} + (n-t){t \choose 2} + 2t\eta_k,
\end{displaymath}
where $\eta_k = 1$, if $k$ is even, and 0 otherwise. Moreover, the only extremal graph is $G_{n,k,t}$.
\end{theorem}

More generally we have,

\begin{theorem} 
For every positive integer $k\geq 3$ and $r \ge 2$, there exists $n_1 \in \mathbb{N}$ such that if $n \geq n_1$,
\begin{displaymath}
\ex(n,S_r,P_k) = \N(S_r,G_{n,k,t}) = t{n - 1\choose r} + (n-t){t \choose r} + 2\eta_k{t \choose r-1},
\end{displaymath}
where $\eta_k = 1$, if $k$ is even, and 0 otherwise. Moreover, the only extremal graph is $G_{n,k,t}$, unless $k$ is even and $t \leq r-2$ in which case the only extremal graphs are $G_{n,k,t}$ and $G_{n,k-1,t}$. 
\end{theorem}



Again, the result follows from a claim about the difference of the values of two consecutive extremal numbers. Let $\displaystyle a_n = \N(S_r,G_{n+1,k,t})-\N(S_r,G_{n,k,t}) = {t \choose r} + t{n-1 \choose r-1}$.


\begin{claim} There exists $n_0 \in \mathbb{N}$ such that for every $n\geq n_0$, $\ex(n+1,S_r,P_k)-\ex(n,S_r,P_k) \leq a_n$ and equality can hold only if either $G_{n+1,k,t}$ is the only extremal graph on $n+1$ vertices or $k$ is even, $t \le r-2$ and the only extremal graphs are $G_{n+1,k,t}$ and $G_{n+1,k-1,t}$.
\end{claim}

\begin{proof}
For any graph $G$, we have that $\N(S_r,G) = \displaystyle \sum_{v\in V(G)}{d(v) \choose r}$. Let $G$ be a $P_k$-free graph with $n+1$ vertices and maximum number of copies of $S_r$; that is, $\N(S_r,G)  = \ex(n+1,S_r,P_k)$.  
 We will consider cases depending on the minimum degree of $G$.

If $\delta(G) > t,$ then every connected component of $G$ must have at most $k$ vertices. So the number of copies of $S_r$ is bounded by $n{k-1 \choose r}$, then we choose $n_0$ such that this number is less than $\N(S_r,G_{n,k,t})$ for $n\geq n_0$.

If $\delta(G) \leq t$, then by removing $v$ a vertex of minimum degree, we remove at most \begin{displaymath}
{d(v) \choose r} + \sum_{u\in N(v)}{d(u)-1 \choose r - 1} \leq {t \choose r} + t{n-1 \choose r-1} 
\end{displaymath}
copies of $S_r$.
Equality can hold only if $\delta(G) = t$ and $t$ vertices have degree $n$, so $G$ contains a complete bipartite graph with color classes of size $t$ and $n+1-t$ such that class of size $t$ is a clique.  Then the characterization of the extremal cases again follows from the maximality of $G$.
\end{proof}


\begin{remark}
By checking more carefully the difference between the number of $r$-stars using $v$ and the number $a_n$, we can find a bound for $n_1$ of order $k^{3/2}$.
\end{remark}


Similarly to before, the same method proves the following two results.

\begin{theorem}
\label{p2cycles}
For every positive integer $k\geq 5$, there exists $n_1 \in \mathbb{N}$ such that if $n \geq n_1$,

\begin{displaymath}
\ex(n,P_2,C_{\geq k}) = \N(P_2,G_{n,k,t}) = t{n - 1\choose 2} + (n-t){t \choose 2} + 2t\eta_k,
\end{displaymath}
where $\eta_k = 1$, if $k$ is even, and 0 otherwise. Moreover the only extremal graph is $G_{n,k,t}$.
\end{theorem}
Or more generally,

\begin{theorem}
For every positive integer $k\geq 5$, there exists $n_1 \in \mathbb{N}$ such that if $n \geq n_1$,
\begin{displaymath}
\ex(n,S_r,C_{\geq k}) = \N(S_r,G_{n,k,t}) =  \N(P_2,G_{n,k,t}) =t{n - 1\choose r} + (n-t){t \choose r} + 2\eta_k{t \choose r-1},
\end{displaymath}
where $\eta_k = 1$, if $k$ is even, and 0 otherwise. Moreover the only extremal graph is $G_{n,k,t}$, unless $k$ is even and $t \leq r-2$ in which case the only extremal graphs are $G_{n,k,t}$ and $G_{n,k-1,t}$.
\end{theorem}

\begin{remark}
For $k=3$, Theorem \ref{p2cycles} also holds. Since $G$ must be a tree and by convexity the number of stars is maximized in a star of $n$ vertices, we have $G_{n,3,1} = K_{1,n-1}$, and this graph has ${n-1\choose r}$ stars. For $k= 4$, a star with a perfect matching or almost perfect matching in the neighborhood of the center vertex maximizes the number of copies of $P_2$,  with ${n-1 \choose 2} + (n-1)$, when $n$ is odd or ${n-1 \choose 2} + (n-2)$, when $n$ is even. Any graph containing the $n$ vertex star maximizes the number of copies of $S_r$ for $r\geq 3$. 
\end{remark}

\subsection{Number of copies of $P_3$}

\begin{theorem}
For every positive integer $k\geq 5$, there exists $n_1 \in \mathbb{N}$ such that if $n \geq n_1$
\begin{displaymath} \ex(n,P_3,P_k) = \N(P_3,G_{n,k,t}) = \frac{3t(t-1)}{2}n^2 + O(n).
\end{displaymath}
Moreover the only extremal graph is $G_{n,k,t}$.
\end{theorem}


Let $\displaystyle a_n =\N(P_3,G_{n+1,k,t})-\N(P_3,G_{n,k,t}) = 2t\left({t-1 \choose 2} + (n-t)(t-1) + \eta_k\right) + t(t-1)(n-2).$\\

As in the previous results it is enough to prove the following claim.

\begin{claim} There exists $n_0 \in \mathbb{N}$ such that for every $n\geq n_0$, 
$\ex(n+1,P_3,P_k) - \ex(n,P_3,P_k) \leq a_{n}$ 
and equality can hold only if  $G_{n+1,k,t}$ is the only extremal graph on $n$ vertices. 
\end{claim}

\begin{proof} Let $G$ be an $(n+1)$-vertex graph with maximum number of copies of $P_3$. We may assume that $\delta(G) \geq 2$. (If a vertex has degree 1, then it is in at most $2(t(n-t) + {t \choose 2} + \eta_k)$ copies of $P_3$.)

If $\delta(G) > t$, then each connected component of $G$ must have size at most $k$ (by Dirac's argument) and so $\N(P_3,G) \leq 3{k \choose 3}\frac{n+1}{k}$. In this case, we can choose $n_0$ such that for $n\geq n_0$, this number is less than $\N(P_3,G_{n,k,t})$.  Thus, we assume that $\delta(G) \leq t$.

Let $v$ be a vertex in $G$ with minimum degree, and consider the copies of $P_3$ beginning containing $v$ as their second vertex. We may suppose $G$ is connected and has enough vertices to apply Theorem~$\ref{connEG}$. 
Then the number of copies of $P_3$ whose second vertex is $v$ is bounded from above by 
\begin{displaymath}
d(v)(d(v)-1)(n-2) - 2(d(v)-2)\left(\binom{d(v)}{2}-e(N(v))\right).
\end{displaymath}
Indeed, the first term is the trivial upper bound $2 \binom{d(v)}{2}(n-2)$ obtained if every pair of neighbors of $v$ could be extended to path of length 3 in any possible way. The subtraction comes from the fact that each non-edge $\{a,b\}$ in the neighborhood of $v$ along with a third neighbor $c \in N(v)$ uniquely forbids 2 copies of $P_3$ namely $cvab$ and $cvba$.  We have bounded from above the number of copies of $P_3$ containing $v$ as a second vertex.  Now we will obtain an estimate on the number of copies of  $P_3$ starting at $v$.
We consider the number of ways to take distinct $u \in N(v), w \in N(u)$ and $x \in N(w)$:\\


 $\displaystyle \sum_{u\in N(v)}\sum_{\substack{w\in N(u)\\ w \not = v}}\Bigg( d(w)-1-\mathbbm{1}_{w\in N(v)}\Bigg) = \sum_{u\in N(v)}\Bigg(\sum_{\substack{w\in N(u)\\ w \not = v}}\Bigg( d(w)\Bigg) - d(u) +1\Bigg) - 2e(N(v))$\\
 
$ \displaystyle = \sum_{u\in N(v)}\Bigg(\sum_{w\in V(G)}\Bigg(d(w)\Bigg)  - \sum_{\substack{w\not\in N(u) \\ w\not=u}}\Bigg(d(w)\Bigg) - d(v)-2d(u)\Bigg) - 2e(N(v))+d(v) $\\

 $ \displaystyle = \sum_{u\in N(v)}\Bigg(2e(G) - \sum_{\substack{w\not\in N(u)\\w\not=u}}\Bigg(d(w)\Bigg) -2d(u)\Bigg) - 2e(N(v)) -d(v)(d(v)-1) $\\

$\displaystyle\leq \sum_{u\in N(v)}\Bigg(2e(G) - 2(n-d(u)) -2d(u)\Bigg) - 2e(N(v)) -d(v)(d(v)-1) $\\

$\displaystyle = 2d(v)(e(G) - n) - 2e(N(v)) -d(v)(d(v)-1), $\\

where the inequality uses that $\delta(G)\geq 2$.  

The above sum is maximized when $d(v) = t$, and to achieve this maximum it is necessary that for every neighbor $u$ of $v$, we have that the non-neighbors of $u$ have degree 2. Moreover, if $t=2$ it is simple to check that to maximize the expression the two neighbors of $v$ must be adjacent. From here it follows that the number of copies of $P_3$ containing $v$ is at most $a_n$. 

When $t\geq 3$, we have the bound (conditioning on whether $v$ is at the beginning or middle of the path)
\begin{displaymath}
2t(e(G) - n) - 2e(N(v)) -t(t-1) + t(t-1)(n-2) - 2(t-2)\left({t \choose 2} - e(N(v))\right),
\end{displaymath}
\begin{displaymath}
\leq 2t(e(G) - n)  -2t(t-1) + t(t-1)(n-2).
\end{displaymath}
From Theorem \ref{connEG} it follows that this number is at most $a_n$.
To obtain equality, in both cases it is necessary that every neighbor of $v$ has full degree and so by maximality we have that $G = G_{n+1,k,t}$.
\end{proof}

\begin{figure}[t]
\centering 
\label{}
\begin{tikzpicture}

\filldraw (0,0) node[align=center,below]{$v$} circle (2pt) -- (1,1) node[align=center,below]{$u$} circle (2pt);
\filldraw (1,1) -- (2,0) node[align=center,below]{$w$} circle (2pt);
\draw[dashed] (0,0)--(2,0);

\filldraw (2,0) -- (3,.75)circle (2pt);
\filldraw (2,0) -- (3,.25)circle (2pt);

\filldraw (2,0) -- (3,-.25)circle (2pt);
\filldraw (2,0) -- (3,-.75)circle (2pt);
\end{tikzpicture} \qquad
\begin{tikzpicture}

\filldraw (0,1) node[align=center,below]{$v$} circle (2pt) -- (1,.5) node[align=center,below]{$u$} circle (2pt);
\filldraw (0,1) -- (-1,.5) node[align=center,below]{$w$} circle (2pt);
\draw[dashed] (-1,.5)--(1,.5);

\draw (2,-1) arc(0:360: 1.2cm and .5cm);
\draw (1,.5) -- (2,-1);
\draw (1,.5) -- (-.4,-1);

\draw (.2,-1) arc(0:360: 1.2cm and .5cm);
\draw (-1,.5) -- (.2,-1);
\draw (-1,.5) -- (-2.2,-1);

\end{tikzpicture} \qquad
\begin{tikzpicture}

\filldraw (0,0) node[align=center,below]{$u$} circle (2pt) -- (1,1) node[align=center,below]{$v$} circle (2pt);
\filldraw (1,1) -- (2,0) node[align=center,below]{$w$} circle (2pt);
\draw[dashed] (0,0)--(2,0);

\filldraw (3,0) node[align=center,below]{$x$}circle (2pt)--(4,.75) node[align=center,below]{$y$}circle (2pt);

\end{tikzpicture}
\caption{Constructing paths using $v$.}
\end{figure}
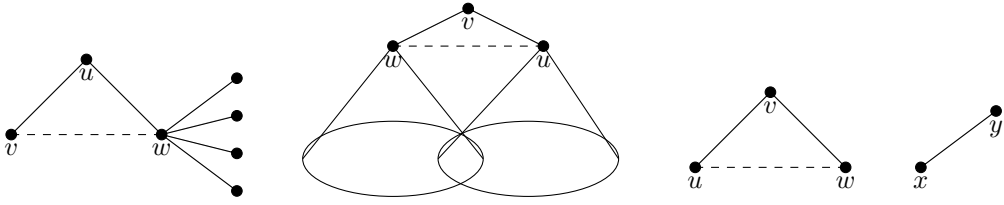

\begin{remark}
For $k = 4$, it is simple to check that the only extremal graph is a balanced double star on $n$ vertices, which has $\lfloor\frac{n-1}{2}\rfloor   \lceil\frac{n-1}{2} \rceil$ copies of $P_3$.
\end{remark}

Next we consider paths of length 4.

\subsection{Number of copies of $P_4$}

\begin{theorem}
\label{P4}
For every positive integer $k\geq 5$, there exists $n_1 \in \mathbb{N}$ such that if $n \geq n_1$, then
\begin{displaymath}
\ex(n,P_4,P_k) = \N(P_4,G_{n,k,t}) = \frac{t(t-1)}{2}n^3 + O(n^2).
\end{displaymath}
Moreover the only extremal graph is $G_{n,k,t}$.
\end{theorem}

Let $\displaystyle a_n \coloneqq\N(P_4,G_{n+1,k,t})-\N(P_4,G_{n,k,t}),$ defined similarly as before. We can check that $\\a_n = 2t\N(P_2,G_{n-1,k-2,t-1}) + 2t(t-1)e(G_{n-2,k-4,t-2}) + {t \choose 2}(n-2)(n-3).$

As in the previous results it is enough to prove the following Claim.

\begin{claim} There exists $n_0 \in \mathbb{N}$ such that for every $n\geq n_0$, 
$\ex(n+1,P_4,P_k) - \ex(n,P_4,P_k) \leq a_{n}$ and equality can hold only if  $G_{n+1,k,t}$ is the only extremal graph on $n$ vertices. 

\end{claim}

\begin{proof} Let $G$ be an $(n+1)$-vertex graph with maximum number of $P_4$ copies. Assume that $\delta(G) \geq 2$. (If a vertex $v$ has degree 1, then it is in at most $2\N(P_2,G_{n,k,t}) < a_n$ copies of $P_4$.)

If $\delta(G) > t$ then each component of $G$ must have size at most $k$ and so $\N(P_4,G) \leq 12{k \choose 4}\frac{n+1}{k}$ so we can choose $n_0$ such that if $n\geq n_0$ this number is less than $\N(P_4,G_{n,k,t}).$

Suppose now that $\delta(G) \leq t$. Let $v$ be a vertex of minimum degree. As before suppose $G$ is connected. To count the number of paths of length 4 starting at $v$, fix
$u\in N(v)$ and let $G'$ be the subgraph of $G$ obtained by removing $v$ and $u$.  Any path $vuu_1u_2u_3$ can be decomposed as the edge $vu$ together with the ordered path $u_1u_2u_3$ in $G'$  so the number of paths of the form $vuu_1u_2u_3$  is at most $2\N(P_2,G'),$ since there are two orderings of any $P_2$.  It is easy to check that $G'$ cannot contain a cycle of length at least $k-1$, otherwise together with the edge $uv$ we would have a copy of $P_k$. Thus, there are two cases: \\

a) Suppose first that $G'$ is does not contain a $C_{k-2}$.  Then $G'$ is $C_{\geq k-2}$-free and so $\N(P_2,G') \leq ex(n-1,P_2,C_{\geq k}) = \N(G_{n-1,k-2,t-1},P_2)$. We will take $n_0$ bigger than the constant from Theorem~\ref{p2cycles}, when $k = 5$, or $k\geq 7$.  When $k=6$ we use the following lemma.

\begin{lemma}
If $H$ is a graph on $n$ vertices containing no cycle of length at least 4, then either $H$ contains a vertex of degree $n-1$ or $\N(P_2,H) < {n-1\choose 2} +2$.
\end{lemma}

\begin{proof}
Suppose $H$ has no vertex of degree $n-1$. If $H$ has degree 1 vertices, then the number of copies of $P_2$ is maximized when all these vertices are adjacent to the vertex of maximum degree, so suppose $H$ has no vertex of degree 1. Then $\sum_{v\in V(H)}d(v) \leq 3(n-1)$ and $2\leq d(v) \leq n-2$.  Therefore by convexity the number of copies of $P_2$ is maximized when there is one vertex of degree $n-2$, one of degree 3 and $n-2$ of degree 2.  This yields ${n-2 \choose 2} + n+1 = {n-1 \choose 2} + 3$, however a graph with such a degree sequence must have a cycle of length bigger than four. Thus we consider the second best is a graph with one vertex of degree $n-2$ and $n-1$ vextex of degree 2 (if possible) which has ${n-1\choose 2} + 1 $ copies of $P_2$.       
\end{proof}

Now according to this lemma for $k=6$, either $\N(P_2,G') < {n-2 \choose 2} +2 = \N(P_2,G_{n-1,4,1})$ or $G'$ has a vertex of degree $n-2$,  call it $w$, and some edges $t$  in $N_{G'}(w)$.  The vertex $u$ cannot be connected to two different edges in $N_{G'}(w)$, otherwise $G$ would contain a $P_6$, and if $u$ is connected to both vertices of one of these edges and to all other vertices of $G'$, then the number of copies of $P_4$ starting with $vu$ would be $(n-2t)(n-3) + 2 +2t \leq (n-2)(n-1) + 4$.   \\

b) Now suppose that $G'$ contains a cycle of length $k-2$, $C$. In this case we have the following.

\begin{claim}
\label{neinC}
If $w$ is a vertex which is not in the cycle and $w\in N(x)$ where $x$ is a vertex of the cycle, then $w$ has at most one neighbor outside of $C$.
\end{claim}

\begin{proof}
Suppose $w_1$ and $w_2$ are two neighbors of $w$. Since $\delta(G) > 1$,  $w_1$ has a neighbor $y$. If $y$ is in $C$, then $C$ together with $w_1ww_2$ is a length $k$ path. If $y$ is outside of $C$, then $C$ together with $ww_1y$ is a $P_k$.
\end{proof}
\begin{figure}
\label{Cyclelemma}
\centering 
\begin{tikzpicture}

\draw (1,0) arc(0:360: 1cm and 1cm);
\filldraw[thick] (-1,1.5) node[align=center, above]{$w_1$} circle (2pt) -- (0,2) node[align=center, above]{$w$} circle (2pt) -- (1,1.5) node[align=center, above]{$w_2$} circle (2pt) ;
\filldraw (0,2) -- (0,1) node[align=center, below]{$x$} circle (2pt);
\filldraw[thick] (xyz polar cs:angle=160,radius=1) node[align=center,left]{$y$} circle (2pt) -- (-1,1.5);
\filldraw (xyz polar cs:angle=140,radius=1) circle (2pt);
\draw[thick] (xyz polar cs:angle=140,radius=1) arc(140:-200:1cm and 1cm);
\end{tikzpicture} \qquad
\begin{tikzpicture}
\draw (1,0) arc(0:360: 1cm and 1cm);
\filldraw[thick] (-1,1.5) node[align=center, above]{$w_1$} circle (2pt) -- (0,2) node[align=center, above]{$w$} circle (2pt);
\filldraw (0,2) -- (1,1.5) node[align=center, above]{$w_2$} circle (2pt) ;
\filldraw[thick] (0,2) -- (0,1) node[align=center, below]{$x$} circle (2pt);
\filldraw[thick]  (-1,1.5)--(0,2.5) node[align=center,left]{$y$} circle (2pt) ;

\filldraw (xyz polar cs:angle=110,radius=1) circle(2pt);
\draw[thick] (xyz polar cs:angle=110,radius=1) arc(110:450:1cm and 1cm);

\end{tikzpicture}
\caption{Lemma of the cycle in Theorem \ref{P4}.}
\end{figure}
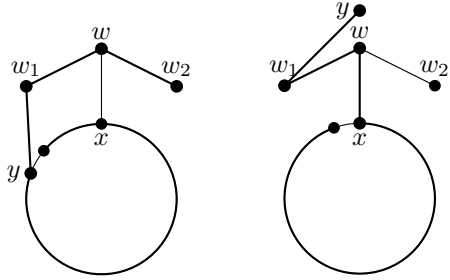


As a corollary we have,
\begin{claim}
\label{neinC2}
If $w$ is a vertex not in the cycle and $w\in N(x)$ where $x$ is a vertex of the cycle, then $d(w) < k$.
\end{claim}

The edge $uv$ is connected to $C$ by some path. If there is an intermediate vertex from $C$ to $uv$, then clearly this path will have length at least $k$ and so this is not possible. Hence $C$ is connected to either $u$ or $v$. If $C$ is connected to $v$, then every neighbor of $u$ must be in $C$ for otherwise we have a $k$ path. If $C$ is connected to $u$, then by Claim \ref{neinC}, all neighbors of $u$, except for $v$, are in $C$. So for every path $vuu_1u_2u_3$ we have that $u_1 \in C$, if $u_2 \in C$, then we have less than $k$ choices for both $u_1$ and $u_2$ and at most $n$ choices for $u_3$. If $u_2$ is not in $C$, then by Claim \ref{neinC2}, since $u_2$ is a neighbor of $u_1 \in V(C)$, we have $d(u_2) < k$ and so there are at most $k$ choices for $u_3$ and less than $n$ choices for $u_2$. Hence we have less than $2k^2n$ such paths in total and  we can take $n_0$ such that if $n\geq n_0$, then this number is less than $2\N(G_{n-1,k-2,t-1},P_2)$.

It follows that the number of paths starting with $v$ is at most $2d(v)\N(P_2,G_{n-1,k-2,t-1}).$

Now if $d(v) \leq t-1$, then the trivial bound on the number of copies of $P_4$ with middle vertex $v$ is $d(v)(d(v)-1){n-2 \choose 2}$ and the bound on the number of $P_4$ cpoies with $v$ as a second vertex is $2d(v)(d(v)-1)e(G)$.  Thus, we would have that the number of copies of $P_4$ using $v$ is less than $a_n$. So we will now suppose $d(v) = t$. To simplify the notation in the following calculations let $S := \sum_{u\in N(v)} d(u)$.

To count paths with $v$ in the middle, we will count in order paths of the form $xuvwy$, where $u,w$ can be any neighbors of $v$ and then we have to choose a neighbor $x$ of $u$ and a neighbor $y$ of $w$ with $y \not = x$. Hence the number of ordered paths with $v$ as the middle vertex is


$\displaystyle\sum_{u\in N(v)}\Bigg(\sum_{\substack{w\in N(v)\\w\not = u}}(d(u)-1-\mathbbm{1}_{u\in N(w)})(d(w)-1-\mathbbm{1}_{u\in N(w)})-|N(u)\cap N(w)|+1\Bigg)$

$\displaystyle\leq \sum_{u\in N(v)}\Bigg(\sum_{\substack{w\in N(v)\\w\not = u}}d(u)d(w) -2d(w) - 2d(u) -\mathbbm{1}_{u\in N(w)}(d(w)+d(u)) + 5\cdot\mathbbm{1}_{u\in N(w)} +n+1\Bigg) $ 

$\displaystyle = S^2 - \Bigg(\sum_{u \in N(v)} d(u)^2\Bigg)  -4(t-1)S -2\Bigg(\sum_{u\in N(v)}d(u)\abs{N(u)\cap N(v)}\Bigg) + 10e(N(v))+t(t-1)(n+1)$ 

$\displaystyle \leq S^2 - \Bigg(\sum_{u\in N(v)}d(u)^2\Bigg) -4(t-1)S - 2\Bigg(\sum_{u\in N(v)}d(u)(d(u)+t-n-1)\Bigg) +t(t-1)(n+6) $

$\displaystyle = S^2 - 3\Bigg(\sum_{u\in N(v)}d(u)^2\Bigg) +(2n-6(t-1))S + t(t-1)(n+6) $

$\displaystyle \leq S^2 - \frac{3S^2}{t} +(2n-6(t-1))S + t(t-1)(n+6) $

$\displaystyle  = \frac{t-3}{t}S^2 +(2n-6(t-1))S  +t(t-1)(n+6),$


where in the first and second inequality we use the fact that for every pair of vertices $x,y$ of the graph $|N(x)\cap N(y)| \geq d(x) + d(y) -n+1 - 2\cdot\mathbbm{1}_{x\in N(y)}$ together with $e(N(v)) \leq {t \choose 2}$. The last inequality was obtain by applying the Cauchy-Schwarz inequality to $\sum_{u\in N(v)}d(u)^2$. 
Since any path can have two distinct orders we divide this expression by 2.

To count the number of paths with $v$ as the second vertex, we will decompose the path $uvwxy$ into $uvw$ together with $e=xy$. First we choose in order two neighbors of $v$, then an edge not using $u,v$ or $w$. There are at most 2 ways to connected the edge to $w$, so the number of these paths is at most 
\begin{displaymath} 
2\displaystyle\sum_{u\in N(v)}\Bigg(\sum_{\substack{w\in N(u)\\w\not = u}}e(G)-d(w)-d(u)-t+2+\mathbbm{1}_{u\in N(w)}\Bigg)
\end{displaymath} 
\begin{displaymath} 
= 2t(t-1)(e(G)-t+2) + 4e(N(v)) -4(t-1)S
\end{displaymath}
\begin{displaymath}
\leq 2t(t-1)(e(G_{n,k,t})+2) + 4{t \choose 2} -4(t-1)S.
\end{displaymath}


By summing the previous bounds, we have that the number of paths using $v$ is at most 
\begin{displaymath}2t\N(P_2,G_{n,k-2,t-1}) + \frac{t-3}{2t}S^2 +(n-7(t-1))S +t(t-1)(n+6) + 2t(t-1)(e(G_{n,k,t})+2) + 4{t \choose 2}. \end{displaymath} 


The value of this expression when $S=tn$ is precisely $a_n$. By considering this expression as a quadratic in $S$, we can check that if $t\geq 3$ the maximum is attained only when $S = nt.$ This means that every neighbor of $v$ must have degree $n$, so this is only possible if $G = G_{n+1,k,t}$. If $t = 2$, the expression attains its maximum when $S = 2n-14$, hence if $S < 2n -28$. This value would be less than $a_n$, but now with the condition $S \geq 2n - 28$. It is simple to check that for $k = 5$, $G$ must be either $K_{2,n-2}$ or $G_{n,5,2}$, and if $k=6$, then $G$ must be $G_{n,6,2}$.
\end{proof}

\section{The number of copies of $P_{k-1}$ in $P_k$-free graphs}
\label{weird case}

If $k$ is odd, it seems likely that the graph  $G_{n,k,t}$ attains the value $\ex(n,P_{k-1},P_k)$.  However, for $k$ even the situation changes. We have that $\N(P_{k-1},G_{n,k,t}) = \Theta(n^{t})$, but there is another graph $H_{n,k}$ such that $\N(P_{k-1},H_{n,k}) = \Theta(n^{t+1})$. In order to define this graph, first for $r\geq 2$ and $a,b$ positive integers, let $S^{(r)}_{a,b}$ be the $(a+b+r)$-vertex graph consisting of a clique on $r$ vertices and two independent sets $A$ and $B$ on $a$ and $b$ vertices, respectively.  Let $v$ be a vertex of the clique and join $v$ to every vertex in $B$, then join every vertex of the clique except $v$ to every vertex in $A$. Let $\mathcal{S}^{(r)}_n$ be the family of all such graphs on $n$ vertices. For even $k$, let $H_{n,k}$ be the graph in $\mathcal{S}^{(t+1)}_n$ which maximizes the number of $P_{k-1}$ copies. In this case we conjecture that the graph $H_{n,k}$ is extremal.

\begin{conjecture}
\label{weird}
If $k$ is even and $k \ge 4$, the extremal number $\ex(n,P_{k-1},P_k)$ is attained by the the graph $H_{n,k}$.
\end{conjecture}

\begin{remark}
For $r\geq 2$ the graphs $S^{(r)}_{a,b}$ are $P_{2r}$-free and have $(r-1)!ba(a-1)\cdots (a-r+2)$ copies of $P_{2r-1}$. In  $\mathcal{S}^{(r)}_n$ this number is maximized when $a$ is roughly  $\frac{r-1}{r}n$, and this maximum approaches $(r-1)!(\frac{(r-1)^{r-1}}{r^r})n^r$ as $n$ tends to infinity. In particular by taking $r=3$ we have a $P_6$-free graph with $8n^3/27 + O(n^2)$ copies of $P_5$.
\end{remark}

\begin{remark}
Note that the only edges of the clique in $S^{(r)}_{a,b}$ that a $P_{2r-1}$ uses are the ones that are incident with the vertex $v$. So we have several graphs for which we conjecture the number of copies of $P_{2r-1}$ is maximal, namely those subgraphs of $S^{(r)}_{a,b}$ formed by removing edges from the clique not incident to $v$. 
\end{remark}
 
Conjecture \ref{weird} can be easily checked for $k=4$, and the following theorem says that this conjecture is also true for $k=6$.

\begin{theorem}
\label{p5p6}
There exists $n_1 \in \mathbb{N}$ such that if $n \geq n_1$, then \begin{displaymath}
\ex(n,P_5,P_6) = \N(P_5,H_{n,6}) = \frac{8n^3}{27} + O(n^2).
\end{displaymath}
\end{theorem}

\begin{proof}

Let $G$ be a $P_6$ free graph, and suppose $n\geq 7$. It is enough to bound the copies of $P_5$ in each connected component, so assume $G$ is connected.

Let $C$ be the largest cycle in $G$ and let $G'$ be the graph obtained by removing $C$ from $G$; clearly $C$ cannot be a 6-cycle, otherwise $G$ would contain a $P_6$. We will consider cases based on the length of $C$. \\

a) Suppose $C$ is a 5-cycle with vertices $v_1,v_2,v_3,v_4,v_5$ appearing consecutively.  Then every vertex in $G'$ is connected to a vertex of $C$. Suppose that $v_1$ has a neighbor in $G'$, if $v_1$ is the only vertex of $C$ with an edge to $G'$, then $\N(P_5,G) < 24n$.  So suppose this is not the case, $v_2$ and $v_5$ cannot have neighbors in $G'$. Thus, without loss of generality, we may suppose $v_3$ has a neighbor in $G'$, then $v_4$ cannot have neighbors in $G'$, also note that $v_2$ cannot be connected with $v_4$ or with $v_5$ and so $G \subseteq G_{n,6,2}$ (where the $v_1$ and $v_3$ take the role of the high degree vertices, and the edge $v_4v_5$ is the only edge that is not incident with one of $v_1$ or $v_3$), hence $\N(P_5,G) = O(n^2)$.   \\


b) Now suppose $C$ is a 4-cycle defined by $v_1,v_2,v_3,v_4$, consecutively. Then $G'$ cannot contain a $P_3$, otherwise by connectivity we would have a path of length at least 6. Consider the set $X$  of vertices of $G'$ that have at least one neighbor in both $C$ and $G'$. Note that if $y \in G'$ is a neighbor of $x\in X$, then $y$ cannot have any other neighbor in $G'$. Also note that the only possible neighbor of $y$ in $C$ is the neighbor of $x$ in $C$ ($y$ cannot have a neighbor in $C$ if $x$ has more than one neighbor in $C$).

If $\abs{X} > 1$, then every vertex in $X$ must be adjacent to the same vertex in $C$, say $v_1$. Then $v_2$ and $v_4$ cannot have a neighbor outside of $C$.  If $v_2$ and $v_4$ are adjacent, then it also holds that $v_3$ cannot have neighbors in $G'$. It is then easy to check that $\N(P_5,G) < 6n$. So suppose $v_2$ and $v_4$ are not connected (see Figure \ref{weirdext}), then every $P_5$ in $G$ is of the form $xyv_1vv_3u$, where $x \in X$, $y$ is a neighbor of $x$, and both of $v,u$ are common neighbors of $v_1$ and $v_3$. If $a = \abs{N(v_1) \cap N(u) \cap N(v)}$ and $b$ is the number vertices in $G'$ with a neighbor in $X$, then we have that $\N(P_5,G) \leq ba(a-1)$ which is half $\N(P_5,S^{(3)}_{a,b})$ but $S^{(3)}_{a,b}$ can have at most one more vertex than $G$.     


\begin{figure}
\label{weirdext}
\centering 
\begin{tikzpicture}

\foreach \xi in {0,60,...,360}{
\pgfmathsetmacro{\z}{\xi+60}
\foreach \yi in {\xi,\z,...,360}{

\filldraw  (xyz polar cs:angle=\xi,radius=1) circle (2.5pt) -- (xyz polar cs:angle=\yi,radius=1) circle (2.5pt);

}
}

\foreach \x in{-1,-0.5,...,1}{
\foreach \y in {0,60,...,360}{

\filldraw (xyz polar cs:angle=\y,radius=1) -- (\x,-2) circle (1.5pt);
}
}
\draw  (1.5,-2) arc (0:360:1.5cm and .3cm) ;
\foreach \y in {0,60,...,360}{
\filldraw (xyz polar cs:angle=\y,radius=1) -- (3,0) circle (2.5pt);
}

\foreach \x in {2,2.5,...,4}
\filldraw (3,0)-- (\x,-2) circle (2pt);

\end{tikzpicture} \qquad
\begin{tikzpicture}
\draw (-1,0) -- (0,1) -- (1,0) -- (-1,0);

\foreach \x in {-2,-1.5,...,2}{
\filldraw (-1,0) circle (2pt) -- (\x,-1)circle (2pt);

}

\foreach \x in {-2,-1.5,...,2}{
\filldraw (1,0) circle (2pt) -- (\x,-1)circle (2pt);

}

\foreach \x in {-1.5,-1,...,1.5}{
\filldraw (0,1) circle (2pt) -- (\x,2)circle (2pt);

}

\end{tikzpicture}
\begin{tikzpicture}

\filldraw (0,1) circle (2pt) -- (-.75,0) circle (2pt);
\filldraw (0,-1) circle (2pt) -- (-.75,0) circle (2pt);
\filldraw (0,1) circle (2pt) -- (.75,0) circle (2pt);
\filldraw (0,-1) circle (2pt) -- (.75,0) circle (2pt);

\draw  (-1.5,1.5) circle (2pt) -- (-1,2) -- (0,1);
\draw (-.5,2) circle (2pt) --  (0,1.5) circle (2pt) -- (.5,2) circle (2pt);

\draw (0,1) circle (2pt) --  (0,1.5) circle (2pt) -- (0,2) circle (2pt);

\draw (1.5,2) --  (1.5,1.5)  -- (1,2);

\draw (-1.5,1.5) -- (0,1) -- (1.5,1.5); 

\draw (0,1) -- (-1.5,0) -- (0,-1) ;
\draw (0,1) --  (1.5,0) -- (0,-1) ;
\draw (0,1) --  (0,-1) ;

\filldraw     (-1.5,1.5) circle (2pt)  (-1,2) circle (2pt) (-.5,2) circle (2pt)  (0,1.5) circle (2pt)  (.5,2) circle (2pt)  (0,2) circle (2pt) (1.5,2) circle (2pt)  (1.5,1.5) circle (2pt) (1,2) circle (2pt) (-1.5,1.5) circle (2pt) (1.5,1.5)circle (2pt) (-1.5,0) circle (2pt) (1.5,0) circle (2pt); 

\end{tikzpicture}
\caption{The family $\s^{(r)}_{a,b}$ from which the conjectured extremal graph $H_{n,k}$ is obtained. And a graph that appears in Case b) of the proof of the Theorem \ref{p5p6}.}
\end{figure}
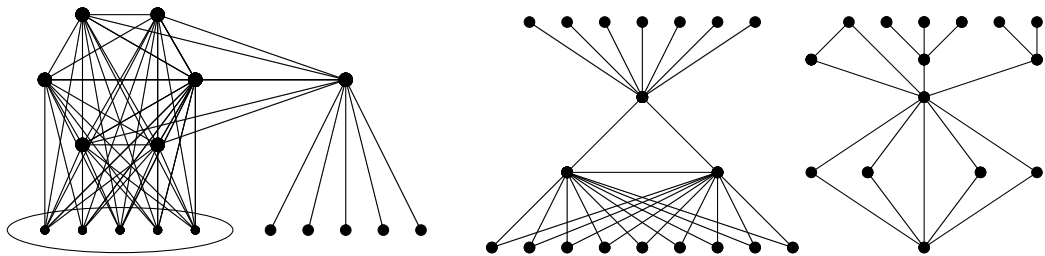

If $X= \{x\}$, then something similar holds, except that both $v_1$ and $v_3$ are allow to be connected to $x$ and there is the extra possibility of $G$ being a subgraph of $S^{(3)}_{a,b}$.


Now suppose $X = \emptyset $. If no two vertices of $C$ share a common vertex in $G'$, then $\N(P_5,G)$ is quadratic. So suppose two non-consecutive vertices, say $v_1$, and $v_3$, share a common neighbor, then it is not possible for the other two vertices in $C$ to have a neighbor in $G'$. Thus, our graph is again a subgraph of $S^{(3)}_{a,b}$. \\

c) Suppose $C$ is a triangle, then every pair of vertices are the end vertices of at most one $P_5$. If two different paths of length 5 have the same end vertices, then either $G$ would contain a cycle of length at least four or a $P_6$. Thus, $\N(P_5,G) < {n\choose 2}$. \end{proof}

\bibliographystyle{abbrvnat}
\bibliography{plpkbib.bib}

\end{document}